\documentclass[12pt]{amsart}

\usepackage{amsmath,amssymb,amscd,verbatim,color}
\newcommand{\mc}{\mathcal}

\newcommand{\sub}{\subseteq}

\newcommand{\subneq}{\subsetneq}

\newcommand{\ol}{\overline}

\newcommand{\lra}{\Leftrightarrow}

\newcommand{\ra}{\Rightarrow}

\newcommand{\sm}{\setminus}

\newcommand{\astdim}{\ast\op{-dim}}
\newcommand{\tdim}{t\op{-dim}}
\newcommand{\tmax}{t\op{-Max}}
\newcommand{\Max}{\op{Max}}
\newcommand{\tspec}{t\op{-Spec}}

\newcommand{\astspec}{\ast\op{-Spec}}
\newcommand{\astmax}{\ast\op{-Max}}

\newcommand{\op}{\operatorname}

\newcommand{\F}{\mc F}
\newcommand{\cP}{\mc P}
\newcommand{\cS}{\mc S}

\newtheorem{theorem}{Theorem}[section]

\newtheorem{lemma}[theorem]{Lemma}

\newtheorem{prop}[theorem]{Proposition}

\newtheorem{cor}[theorem]{Corollary}

\newtheorem{remark}[theorem]{Remark}

\newtheorem{example}[theorem]{Example}

\theoremstyle{definition}

\begin{document}

	\title{Star stability and star regularity for Mori domains}

%\subjclass{Primary: ; Secondary: .}


\author{Stefania Gabelli and Giampaolo Picozza
}

\address{Dipartimento di Matematica, Universit\`{a} degli Studi Roma
Tre,
Largo S.  L.  Murialdo,
1, 00146 Roma, Italy}

\email{gabelli@mat.uniroma3.it, picozza@mat.uniroma3.it
}

%%%%%%%%%%%%%%%%%%%%%%%%%%%%%%%%%%%%%%%%%%

%%%%%%%%%% ABSTRACT

\begin{abstract} In the last few years, the concepts of stability and Clifford regularity have been fruitfully extended by using star operations. In this paper we study and put in relation these properties for  Noetherian and Mori domains, substantially improving several results present in the literature. \end{abstract}

%%%%%%%%%%%%%%%%%%%%%%%%%%%%%%%%%%%%%%%%%%%

\maketitle

\section*{Introduction}

In this paper, we study stability and Clifford regularity of Noetherian and Mori domains with respect to star operations, substantially improving several results present in the literature and answering some questions left open.

Recall that a domain $R$ is \emph{stable} if each nonzero ideal $I$ of $R$ is invertible in its endomorphism ring $E(I):=(I:I)$. Stable
domains have been thoroughly investigated by B. Olberding  \cite{O3, O5, O1, O2}.

Stability implies Clifford regularity. A domain is called \emph{Clifford regular} if each nonzero ideal $I$ of $R$ is \emph{von Neumann regular}, that is $I=I^2J$ for some ideal $J$.
This concept was studied for orders in quadratic fields by P. Zanardo and U. Zannier in \cite{ZZ}. Later S. Bazzoni and L. Salce  proved that all valuation domains are Clifford regular \cite{BS} and S. Bazzoni deepened the study of Clifford regularity  in \cite{B1, B2, B3, B4}.

Stability with respect to semistar operations was introduced and studied by the authors of this paper in \cite{GP}.

The first attempt to extend the notion of Clifford regularity in the setting of star operations is due to S. Kabbaj and A. Mimouni, who studied Clifford $t$-regularity for P$v$MDs and Mori domains \cite{KM1, KM2, KM3, KM4}.
Successively F. Halter-Koch, in the language of ideal systems, introduced Clifford $\ast$-regularity for any star operation of finite type \cite{HK} and the authors of this paper continued this study in \cite{GP2}.

Here we mainly  consider the cases where $\ast= d, w, t$.
In fact the most interesting results on star stability and star regularity were obtained in \cite{GP} and \cite {GP2} for star operations spectral and of finite type. In addition, if $\ast$ is spectral and of finite type, $\ast$-regularity implies $\ast=w$ \cite[Corollary 1.7]{GP2}; in particular, if $R$ is Clifford regular, then $w=d$.

Definitions are given in Section 1.

In Section 2, we prove that  $t$-regularity and $t$-stability extend to $t$-compatible Mori overrings (Proposition \ref{overmori}). We also show
that for Mori domains $t$-regularity, $t$-stability and $w$-stability are $t$-local properties, while $w$-regularity is a  $t$-local property for strong Mori domains (Proposition \ref{locMori}).

In Section 3, we prove that Noetherian $\ast$-stable domains have $\ast$-dimension one and strong Mori $t$-stable domains have $t$-dimension one (Corollary \ref{corN2}), extending the well known result that Notherian stable domains are one-dimensional and a result of Kabbaj and Mimouni \cite[Lemma 2.7]{KM4}. This allows us to show that, for strong Mori domains, $w$-regularity and $w$-stability are equivalent (Corollary \ref{onedim1}), while $t$-regularity and $t$-stability are equivalent only in dimension one (Corollary \ref{onedim2}).
 In relation to a question of Kabbaj and Mimouni \cite[Question 2.11(3)]{KM4},
we also prove that the $w$-integral closure of a strong Mori Boole $w$-regular domain is a unique factorization domain (Proposition \ref{Q2}).

In Section 4, we consider the Mori case. We are unable to extend all the results obtained for the $t$-operation in the Noetherian case; however  we show that $w$-stable Mori domains have $t$-dimension one (Proposition \ref{dimwstable}) and that in $t$-dimension one $w$-regularity and $w$-stability are equivalent (Proposition \ref{wmori}). As a consequence, we are able to give an answer to a question posed by Kabbaj and Mimouni in \cite[page 633]{KM1} (Corollary \ref{Q1}): namely we show that a Mori $w$-regular domain has $t$-dimension one if and only if its $w$-integral closure is a Krull domain (Corollary \ref{Q1}).

%%%%%%%%%%%%%%%%%%%%%%%%%%%%%%%%%%%%%%
\section{Preliminaries and notation}

Throughout all the paper, $R$ will be an integral domain and $K$ its field of fractions. If $I$ is a nonzero fractional ideal of $R$, we call $I$ simply an \emph{ideal} and if $I\sub R$ we say that $I$ is an \emph{integral ideal}.

We assume  that the reader is familiar with the properties of star operations (see for example \cite[Sections 32, 34]{g1}). Occasionally we will also consider semistar operations; standard material about semistar
operations can be found in \cite{EFP}. We just recall some basic notions that will be used in the paper.

By $\ol{\mc F}(R)$ we denote the set of nonzero $R$-submodules of $K$ and by $\mc F (R)$  the
semigroup of all ideals of $R$.  A \emph{semistar operation} (respectively, a \emph{star operation}) $\ast$ on $R$ is a map
$\ol{\mc F}(R)\rightarrow \ol{\mc F}(R)$ (respectively, $\mc F(R)\rightarrow \mc F(R)$), $I\mapsto I^{\ast}$, such that the following
conditions hold for each $0\not=a\in K$ and for each $I$, $J\in \ol{\mc F}(R)$ (respectively, $\mc F(R)$):

\begin{itemize}
\item[(i)]  $(aI)^\ast = aI^\ast$ (respectively, $(aI)^\ast = aI^\ast$ and $R=R^\ast$);

\item[(ii)] $I \subseteq I^\ast$, and $I \subseteq J \Rightarrow I^\ast
\subseteq J^\ast$;

\item[(iii)] $I^{\ast\ast} = I^\ast$.
\end{itemize}

If $\ast$ is a semistar operation on $R$ such that $R^\ast=R$, $\ast$ is called a \emph{(semi)star operation}
on $R$ and its restriction to the set of ideals $\mc F(R)$
is a star operation on $R$, still denoted by $\ast$. Conversely, any star operation $\ast$ on $R$ can be extended to a (semi)star operation by setting $I^\ast=K$ for all $I\in \ol{\mc F}(R)\sm \mc F(R)$.

We will be mainly concerned with star operations and (semi)star operations.

If $\ast$ is a semistar operation on $R$ and $D$ is an overring of $R$,
 the restriction of $\ast$ to the set of $D$-submodules of $K$ is a semistar operation on $D$,  here denoted by $\ast_{\vert D}$ or by $\dot{\ast}$ when no confusion arises. When $D^ \ast =D$,  $\dot{\ast}$ is a (semi)star operation on $D$
\cite[Proposition 2.8]{fl01}.
Note that $\dot{\ast}$ shares many properties with $\ast$ (see
for instance \cite[Proposition 3.1]{giampa}); for example,
if $\ast$ is of finite type then $\dot{\ast}$ is of
finite type \cite[Proposition 2.8]{fl01}.

To any semistar operation $\ast$, we can associate a semistar operation of finite type $\ast_f$, defined by
$I^{\ast_f}:=\bigcup \{J^\ast \, \vert \, J \in {\mc F}(R)
\mbox{ finitely generated and } J \subseteq I\}$, and a semistar operation spectral and of finite type $\widetilde{\ast}$, defined by
$I^{\widetilde{\ast}}:= \bigcap_{M\in \ast_f\op{-Max}(R)} IR_M$, for all $I\in \ol{\mc F}(R)$.

If $\ast$ is a (semi)star operation, an ideal $I$ is a \emph{$\ast$-ideal} if $I=I^{\ast}$ and $I$ is called
\emph{$\ast$-finite} (or \emph{of finite type}) if $I^\ast=J^{\ast}=J^{\ast_f}$ for some finitely
generated ideal $J\in \ol{\mc F} (R)$.

A \emph{$\ast$-prime ideal} is a prime ideal which is also a $\ast$-ideal and a \emph{$\ast$-maximal ideal} is a
$\ast$-ideal maximal in the set of proper integral $\ast$-ideals of $R$. We denote by
$\astspec(R)$ (respectively, $\astmax(R)$) the set of $\ast$-prime
(respectively, $\ast$-maximal) ideals of $R$. If $\ast$ is a (semi)star
operation of finite type, by Zorn's lemma each $\ast$-ideal is contained in
a $\ast$-maximal ideal, which is prime. In this case, $R=\bigcap _{M\in \astmax(R)}R_{M}$. We say that $R$
has \emph{$\ast$-finite character} if each nonzero element of $R$ is contained in
at most finitely many $\ast$-maximal ideals.

When $\ast $ is of finite type, a minimal prime of a $\ast$-ideal is
a $\ast$-prime. In particular, any minimal prime over a nonzero
principal ideal (in particular any height-one prime) is a $\ast$-prime, for any (semi)star operation $\ast$ of finite
type.
The \emph{$\ast$-dimension} of $R$ is the supremum of the lengths of the chains of prime ideals
$(0)\sub P_1\sub \dots \sub P_n\sub \dots$, where $P_i\in \astspec(R)$.

The \emph{identity} is a (semi)star operation denoted by $d$, $I^{d}:=I$
for each $I\in  \ol{\mc F}(R)$.
Two nontrivial (semi)star operations which have been
intensively studied in the literature are the $v$-operation and the $t$-operation.
The \emph{$v$-closure} of $I$ is defined by setting $
I^{v}:=(R:(R:I))$, where for any $I, J\in \ol{\mc F}(R)$ we set $(J\colon I):=\{x\in K\,:\, xI
\subseteq J\}$. A $v$-ideal of $R$ is also called a \emph{divisorial ideal}.
The \emph{$t$-operation} is the (semi)star operation of finite type associated to $v$ and is therefore defined by setting
$I^t:= \bigcup \{J^v \, \vert \, J \mbox{ finitely generated and } J \subseteq I\}$.
The (semi)star operation spectral and of finite type associated to $v$ is called the \emph{$w$-operation} and is
defined by setting $I^w:=\bigcap_{M\in \tmax(R)} IR_M$.

If $\ast_1$ and $\ast_2$ are (semi)star operations on $R$, we say
that $\ast_1 \leq \ast_2$ if $I^{\ast_1} \subseteq
I^{\ast_2}$, for each $I \in \ol{\mc F} (R)$. This is equivalent to the
condition that $(I^{\ast_1})^{\ast_2}= (I^{\ast_2})^{\ast_1} =
I^{\ast_2}$.
If $\ast_1 \leq \ast_2$, then $(\ast_1)_f \leq (\ast_2)_f$ and $\widetilde{\ast_1} \leq \widetilde{\ast_2}$.
Also, for each (semi)star
operation $\ast$, we have $d\leq \ast\leq v$ (so that $\ast_f\leq t$ and $\widetilde{\ast}\leq w$)
 and $\widetilde{\ast} \leq \ast_f \leq \ast$ (so that $w\leq t\leq v$).

 For any (semi)star operation $\ast$, the set of $\ast$-ideals of $R$, denoted by $\mc F_\ast(R)$, is a
semigroup under the $\ast$-multiplication, defined by $(I, J)\mapsto (IJ)^\ast$, with
unit $R$. We say that an ideal $I\in \mc F (R)$ is  \emph{$\ast$-invertible} if $I^\ast$ is
invertible in $\mc F_\ast(R)$, equivalently
$(I(R:I))^\ast=R$.
If $\ast$ is a (semi)star operation of finite type, then $I$ is $\ast$-invertible if and only if $I$ is $\ast$-finite and $I^\ast R_M$ is principal for each $M\in \astmax(R)$ \cite[Proposition 2.6]{K}.

For a  semistar operation $\ast$ on $R$, if $I$ is an ideal of $R$ and $E:=E(I^\ast):=(I^\ast:I^\ast)$, it is easy to
see that $E^\ast=E$. Thus the
restriction of $\ast$ to the set of $E$-submodules of $K$ (denoted by $\dot{\ast}:=\ast_{\vert E}$)
is a (semi)star operation on $E$.
 We say that
an ideal $I$ of $R$ is \emph{$\ast$-stable} if $I^\ast$ is
$\dot{\ast}$-invertible in $E$ and that $R$ is
\emph{$\ast$-stable} (respectively, \emph{finitely $\ast$-stable}) if each ideal (respectively, each finitely generated ideal) of $R$ is $\ast$-stable. We also say that $I$ is \emph{strongly $\ast$-stable} if $I^\ast$ is principal in $E$ and that $R$ is \emph{strongly $\ast$-stable} if each ideal is strongly $\ast$-stable.

If $\ast$ is a star operation on $R$, denoting by $\cP(R)$ the group of principal ideals of $R$, the quotient semigroup $\cS_\ast(R):=\F_\ast(R)/\cP(R)$ is called the \emph{$\ast$-Class semigroup} of $R$.

We say that $R$ is \emph{Clifford $\ast$-regular}, or simply \emph{$\ast$-regular}, if the semigroup $\cS_\ast(R)$ is Clifford regular. This means that  each class $[I^\ast]\in \cS_\ast(R)$ is (von Neumann) regular. Note that this is equivalent to saying that each ideal $I^\ast$ is (von Neumann) regular in $\F_\ast(R)$, that is $I^\ast=(I^2J)^\ast$, for some nonzero ideal $J$ of $R$; in this case necessarily $(IJ)^\ast=(I(E(I^\ast):I)))^\ast$, so that $\ast$-stability implies $\ast$-regularity.
 If  $[I^\ast]$ is regular, we say that $I$ is \emph{$\ast$-regular}. If $[I^\ast]$ is idempotent, that is $[I^\ast]=[(I^2)^\ast]$ (equivalently $(I^2)^\ast=xI^\ast$  for a nonzero $x\in K$) we say that $I$ is \emph{Boole $\ast$-regular}  and if each $[I^\ast]$ is idempotent we say that $R$ is \emph{Boole $\ast$-regular}.
Clearly Boole $\ast$-regularity implies Clifford $\ast$-regularity. More precisely, we have the following relations.
\begin{prop} \label{prop1}\cite[ Proposition 1.5]{GP2} Let $I$ be an
ideal of $R$ and, for a star operation $\ast$ on $R$, set $E:=E(I^\ast)$.
\begin{enumerate}
\item[(1)] If $I$  is $\ast$-stable, then  $I$  is $\ast$-regular. Hence a $\ast$-stable domain is Clifford $\ast$-regular.

\item[(2)] If $I$ is $\ast$-regular, then  $I^\ast$ is $v_E$-invertible in $E$ and if, in addition, $I$ is finitely generated, then $I^\ast$ is  $t_E$-invertible in $E$ (where $v_E$ and $t_E$ denote respectively  the $v$-operation and the $t$-operation on $E$).

\item[(3)] $I$ is strongly $\ast$-stable if and only if $I$ is Boole $\ast$-regular and $\ast$-stable.
Hence a strongly $\ast$-stable domain is precisely a Boole $\ast$-regular $\ast$-stable domain.
\end{enumerate}
\end{prop}

Note that if $\ast_1 \leq \ast_2$, then $\ast_1$-regularity (respectively, $\ast_1$-stability) implies $\ast_2$-regularity (respectively, $\ast_2$-stability).

For our purposes, it will be useful to work in the more general context of $\ast$-Noetherian domains.
Given a star operation $\ast$ on $R$, $R$ is called \emph{$\ast$-Noetherian} if it satisfies the ascending chain condition on $\ast$-ideals, or equivalently, if every $\ast$-ideal is $\ast$-finite. This implies that if $R$ is $\ast$-Noetherian then $\ast$ is of finite type. Note that if $\ast_1 \leq \ast_2$, then $\ast_1$-Noetherianity implies $\ast_2$-Noetherianity.

 Clearly when $\ast = d$ a $\ast$-Noetherian domain is just a Noetherian domain. When $\ast = v$ or $t$,  a $\ast$-Noetherian domain is called a \emph{Mori domain} and when $\ast = w$ it is called a \emph{strong Mori} domain.
We recall that $R$ is a strong Mori domain if and only if it is a  Mori domain such that $R_M$ is Noetherian for each $M\in \tmax(R)$.
For the main properties of Mori domains we refer  to the survey article \cite{Bar}.

For technical reasons, the most interesting results on star stability and star regularity were obtained in \cite{GP} and \cite {GP2} for star operations spectral and of finite type.  In addition, we proved that, for $\ast$ of finite type, either $\ast$-stability or $\widetilde{\ast}$-regularity implies that $\widetilde{\ast}=w$ (\cite[Corollary 1.6]{GP} and \cite[Corollary 1.7]{GP2}); in particular, if $R$ is Clifford regular, then $w=d$. This implies that a Clifford regular strong Mori domain is indeed Noetherian.

We observe that Kabbaj and Mimouni in \cite{KM4} defined a $t$-stable domain as a domain with the property that each $t$-ideal $I$ is stable, that is invertible in its endomorphism ring $E(I)$. However this condition is generally stronger than the usual definition, in fact any Noetherian integrally closed domain $R$ is $t$-stable (since each $t$-ideal $I$ of $R$ is $t$-invertible in $R=(I:I)=:E(I)$), but a $t$-ideal of $R$ need not be invertible \cite[Example 2.9]{KM4}.  However we will show in Corollary \ref{lemmaN1} that the two definitions are equivalent in the Noetherian one-dimensional case. On the other hand, denoting by $t_E$ the $t$-operation on $E:=E(I)$, we have that a $t$-regular $t$-ideal $I$ of a Mori domain is always $t_E$-invertible in $E$. This follows from  Proposition \ref{prop1}(2), since  $I=J^t$, with $J$ finitely generated.

%%%%%%%%%%%%%%%%%%%%%%%%%%%%%%%%%%%%%%

\section{Overrings}

Let $R\sub D$ be an extension of domains. If $\ast$ denotes the $w$- or the $t$-operation on $R$ and $\ast_D$ denotes the respective operation on $D$, we say that the extension is \emph{$\ast$-compatible} if $(ID)^{\ast_D}=(I^\ast D)^{\ast_D}$ for each ideal (equivalently, finitely generated ideal) $I$  of $R$. Also recall that $D$ is \emph{$t$-linked} over $R$ if $(Q\cap R)^t \subneq R$ for each prime $t_D$-ideal of $D$ with $Q\cap R\neq (0)$. By \cite[Proposition 3.10]{Jesse}, the extension $R\sub D$ is $w$-compatible if and only if $D$ is $t$-linked over $R$.

We recall that, viewing $\ast$ as a (semi)star operation on $R$, an overring $D$ is $\ast$-compatible over $R$ if and only if $D=D^\ast$ (see for example \cite[Corollary 2.2]{GP2}). This implies that a $t$-compatible extension is also $w$-compatible (i.e., $t$-linked).

Flat extensions and generalized rings of fractions are examples of $t$-compatible extensions.  In addition, the endomorphism rings of $w$, $t$ or $v$-ideals of $R$ are $t$-linked.

\begin{prop} \label{overmori} Let $R$ be a domain and let $D$ be a $t$-compatible Mori overring of $R$.
 If  $R$ is Clifford $t$-regular (respectively, Boole $t$-regular, $t$-stable, strongly $t$-stable), then $D
 $ is Clifford $t_D$-regular (respectively, Boole $t_D$-regular, $t_D$-stable, strongly $t_D$-stable).  \end{prop}
 \begin{proof}
  If $D$ is Mori each ideal of $D$ is $t_D$-finite, that is of type $I^{t_D}$ with $I$ finitely generated.  Since $R$ and $D$ have the same field of fractions, $I$ is an ideal of $R$.  By \cite [Lemma 2.4]{GP2}, it follows that if  $I$ is Clifford $t$-regular (respectively, Boole $t$-regular, $t$-stable, strongly $t$-stable) as an ideal of $R$, then it is Clifford $t_D$-regular (respectively, Boole $t_D$-regular, $t_D$-stable, strongly $t_D$-stable) as an ideal of $D$.  \end{proof}

\begin{cor} \label{quotmori}  Let $R$ be a Mori domain and let $D$ be a generalized ring of fractions of $R$.
 If  $R$ is Clifford $t$-regular (respectively, Boole $t$-regular, $t$-stable, strongly $t$-stable), then $D
 $ is Clifford $t_D$-regular (respectively, Boole $t_D$-regular, $t_D$-stable, strongly $t_D$-stable).
\end{cor}
\begin{proof}
It follows from Proposition \ref{overmori}, because a generalized ring of fractions of a Mori domain is a $t$-compatible Mori domain.
\end{proof}

\begin{remark} \rm A $t$-linked overring of a Mori domain need not be Mori. In fact, if $R$ is a Mori domain and each $t$-linked overring of $R$ is Mori, then $R$ has $t$-dimension one \cite[Proposition 2.20]{DHLZ}. The converse holds for strong Mori domains \cite[Theorem 3.4]{FC}.

However a $t$-compatible  fractional overring of a Mori  domain $R$ is a Mori  domain. In fact, by the next Proposition \ref{overnoeth}, any fractional overring $D$ of $R$ is $\dot{t}$-Noetherian. If, in addition, $D$ is $t$-compatible (i.e., $D^t=D$), $\dot{t}$ is a (semi)star operation on $D$. Hence $\dot{t}\leq t_D$ and so $D$ is also $t_D$-Noetherian, that is Mori. We also observe that
each $t$-compatible fractional overring $D$ of a domain $R$ is of type $E(I^t):= (I^t:I^t)$, for some ideal $I$ of $R$. To see this, note that if $D=D^t$ is an ideal of $R$, we can write $D=x^{-1}I^t$ for some (integral) ideal $I$ of $R$. Whence $D=(D:D)=(I^t:I^t)$.
\end{remark}

We recall that for any domain $R$ and $\ast=d, w, t$, if $R$ is Clifford $\ast$-regular (respectively, Boole $\ast$-regular, $\ast$-stable, strongly $\ast$-stable), each ring of fractions of $R$ has the same property \cite[Corollary 2.6(a)]{GP2}. We now prove a converse for Mori domains.

The following lemma is an easy calculation and follows from the fact that each ideal of a Mori domain is $t$-finite.

\begin{lemma} \label{lemma3} Let $R$ be a Mori domain and let $R_S$ be a ring of fractions of $R$. Then, denoting by $t_S$ the $t$-operation on $R_S$, for any two ideals $J$ and $I$ of $R$, we have
$(J^t:I^t)R_S=(J^tR_S:I^tR_S)$ and $I^tR_S=(IR_S)^{t_S}$.
 \end{lemma}

\begin{prop} \label{locMori} Let $R$ be a Mori domain. Then:
\begin{itemize}
\item[(1)]  $R$ is Clifford $t$-regular (respectively, $t$-stable, $w$-stable) if and only if $R_M$ is Clifford $t_M$-regular (respectively, $t_M$-stable, stable) for each $M\in \tmax(R)$;
\item[(2)]  If $R$ is Clifford $w$-regular, then $R_M$ is Clifford regular, for each $M\in \tmax(R)$.
If, in addition, $R$ is strong Mori, $R$ is Clifford $w$-regular if and only if $R_M$ is Clifford regular for each $M\in \tmax(R)$.
\end{itemize}
\end{prop}
\begin{proof} (1) Assume that $R_M$ is $t_M$-stable for each $M\in \tmax(R)$. Then, if $I$ is a nonzero ideal of $R$, by applying Lemma \ref{lemma3}, we get
$(I(I^t:I^2))^tR_M= (I(I^t:I^2)R_M)^{t_M}= (IR_M((IR_M)^{t_M}:(IR_M)^2))^{t_M}=((IR_M)^{t_M}:(IR_M)^{t_M}) =(I^t:I^t)R_M$.
Whence $(I(I:I^2))^t=(I^t:I^t)$ and $R$ is $t$-stable.

In the same way, if $R_M$ is $t_M$-regular, we get $I^tR_M=(IR_M)^{t_M}=((IR_M)^2((IR_M)^{t_M}:(IR_M)^2))^{t_M}=((I^t)^2R_M(I^tR_M:(I^2)^tR_M))^{t_M}=(I^2(I^t:I^2))^tR_M$, for each  $M\in \tmax(R)$ and so $I^t=(I^2(I^t:I^2))^t$.

Conversely, if $R$ is $t$-regular (respectively, $t$-stable), $R_M$ is $t_M$-regular (respectively, $t_M$-stable) by  Corollary \ref{quotmori}.

The result  for  $w$-stability follows from
\cite[Corollary 1.10]{GP}, since a Mori domain has the $t$-finite character.

(2) If $R$ is Clifford $w$-regular, $R_M$ has the same property by \cite[Corollary 2.6]{GP2}(a). But since $R$ is Mori, $MR_M$ is a $t_M$-ideal (Lemma \ref{lemma3}) and so the $w$-operation is the identity on $R_M$.

If $R$ is strong Mori and $R_M$ is Clifford regular for each $M\in \tmax(R)$, to show that $R$ is Clifford $w$-regular we can proceed as in (1), recalling that if $R$ is strong Mori each $w$-ideal is $w$-finite and so $(I^w:I^2)R_M=(IR_M:I^2)$, for each ideal $I$ and $M\in \tmax(R)$. \end{proof}

We now show that a Mori Clifford $w$-regular domain is $t$-stable, so that for Mori domains

\centerline{$w$-stable $\ra$ $w$-regular $\ra$ $t$-stable $\ra$ $t$-regular.}

\smallskip
The following is an example of a Mori $t$-stable domain that is not ($w$)-stable (this answers the question in \cite[Remark 1.7(1)]{GP}).
We will see in Section 2 that a Noetherian Clifford $t$-regular domain of $t$-dimension strictly greater than one cannot be $w$-stable (equivalently $w$-regular).

\begin{example} \rm Let $R$ be an integrally closed pseudo-valuation domain arising from a pullback diagram of type
$$  \begin{CD}
        R   @>>>    k\\
        @VVV        @VVV    \\
        V  @>\varphi>>  K:= \frac VM\\
        \end{CD}$$
           where $V:=(V, M)$ is a DVR and $R \neq V$. Note that $R$ is Mori one-dimensional \cite[Theorem 2.2]{Bar}, so that $d=w$ and $t=v$ on all the ideals of $R$.

      Since an integrally closed stable domain is Pr\"ufer \cite[Proposition 2.1]{Rush}, then $R$ is not stable. However $R$ is (strongly) $t$-stable; in fact, for each divisorial non-principal ideal $J^v$ of $R$, we have that $J^v=JV$ is an ideal of $V$ \cite[Proposition 2.14]{HH} and so it is principal in $V=(J^v:J^v)$.
\end{example}

\begin{lemma} \label{lemma4} In a Mori finitely stable domain $R$ every $t$-ideal is stable. In particular $R$ is $t$-stable.
\end{lemma}
 \begin{proof}  Let $I=J^t$, with $J$ finitely generated, a $t$-ideal of $R$. Since $R$ is finitely stable, $J$ is invertible in $E(J):=(J:J)$. Since $E(J)\sub E(I):=(I:I)$, $JE(I)$ is invertible in $E(I)$. Since $(JE(I))^t=(J^tE(I))^t=I$, we get
 $E(I)= JE(I)(E(I):JE(I))= JE(I)(E(I):(JE(I))^t)= JE(I)(E(I) : I) \subseteq I(E(I):I) \sub E(I)$.
It follows  that $I$ is invertible in $E(I)$.
 \end{proof}

\begin{prop}  \label{mori2} Let $R$ be a Mori domain.  If $R$ is Clifford $w$-regular, then every $t$-ideal of $R$ is stable. In particular $R$ is $t$-stable.
\end{prop}
\begin{proof}
In a Mori domain, the property that each $t$-ideal is stable is a $t$-local property. In fact, assume that $I^tR_M$ is stable, for each $M\in \tmax(R)$. By using Lemma \ref{lemma3} we have $E(I^t)R_M=E(I^tR_M)= I^tR_M(E(I^tR_M):I^tR_M)=I^t(E(I^t):I^t)R_M$, for each $M\in \tmax(R)$. So that $E(I^t)= I^t(E(I^t):I^t)$.
Thus it is enough to show that every $t_M$-ideal of $R_M$ is stable, for each $M\in \tmax(R)$. By Proposition \ref{locMori}(2), we know that $R_M$ is Clifford regular; hence $R_M$ is finitely stable \cite[Proposition 2.3]{B3}.  Since $R_M$ is Mori, we conclude by Lemma \ref{lemma4}.
\end{proof}

%%%%%%%%%%%%%%%%%%%%%%%%%%%%%%%%%%%%%%%%%%%%%%%%%%%%%%%%%%
\section{Noetherian and strong Mori domains}

A Clifford (indeed  Boole) $t$-regular  Noetherian domain need not be $t$-stable. In fact Kabbaj and Mimouni gave an example of a Boole $t$-regular local Noetherian domain whose maximal ideal is divisorial of height two \cite[Example 2.4]{KM4}, while as we will see in a moment a $t$-stable Noetherian domain has $t$-dimension one. This last fact was proved in \cite[Lemma 2.7]{KM4} under the stronger hypothesis that each $t$-ideal of $R$ is stable.

We work in the more general context of $\ast$-Noetherian domains.

\begin{prop} \label{overnoeth}
Let $R$ be a domain, $\ast$ a star operation on $R$, $D$ a fractional overring of $R$. If $R$ is $\ast$-Noetherian, then $D$ is $\dot{\ast}$-Noetherian.
\end{prop}
\begin{proof}
A $\dot{\ast}$-ideal of $D$ is in particular a (fractional) $\ast$-ideal of $R$. So the ascending chain condition on $\ast$-ideals of $R$ implies the ascending chain condition on $\dot{\ast}$-ideals of $D$.
\end{proof}

Recall that, for a star operation $\ast$ on $R$, the \emph{$\ast$-integral
closure} of $R$ is the integrally closed overring of $R$ defined by
$R^{[\ast]}:= \bigcup \{(J^\ast:J^\ast); \, J\in{\mc F}(R) \mbox{ finitely generated}\}$ \cite{fl01}.
When $\ast=d$ is the identity, we obtain the integral closure of $R$, here denoted by $R^\prime$. If  $\widetilde{R}:= \bigcup \{(I^v:I^v); \, I\in {\mc F}(R) \}$ is the \emph{complete integral closure} of $R$, we have $R\sub R^\prime \sub R^{[\ast]}\sub \widetilde{R}$.

As shown in \cite[Theorem 4.1]{GP2}, when $\ast = \widetilde{\ast}$,  each pair $R, D$ with $R\sub D \subseteq R^{[\ast]}$ satisfies a star version of Lying Over, Going Up and Incomparability. It follows that the $\ast$-dimension of $R$ and the $\dot{\ast}$-dimension of $R^{[\ast]}$ are the same \cite[Corollary 4.2]{GP2}.

\begin{prop} \label{lemmaN2bis} Let $\ast$ be a star operation on a  domain $R$ and assume that $R$ is $\widetilde{\ast}$-Noetherian. If  each $\ast$-maximal ideal $M$ of $R$ is a $t_E$-invertible $t_E$-ideal in $E:=E(M)$, then $\astdim(R)=1=\tdim(R)$.
\end{prop}
\begin{proof}
We adapt the proof of \cite[Lemma 2.7]{KM4}.
Assume that $R$ has $\ast$-dimension greater than one. Note that $\ast$ is of finite type (since $R$ is $\widetilde{\ast}$-Noetherian and so also $\ast$-Noetherian). Let  $P$ be a height-one prime ideal of $R$. Since $P$ is a $t$-ideal, it is also a $\ast$-ideal. So, there exists a $\ast$-maximal ideal $M$ of $R$ containing $P$. Assume $M \neq P$ and let $E := E(M) := (M:M)$. Since $M$ is a $\ast$-ideal, it is also a $\widetilde{\ast}$-ideal, so it is $\widetilde{\ast}$-finite, i.e., $M = J^{\widetilde{\ast}}$ for some finitely generated ideal $J$ of $R$. Thus $E = (J^{\widetilde{\ast}} : J^{\widetilde{\ast}}) \subseteq R^{[\widetilde{\ast}]}$. So, by $\widetilde{\ast}$-GU, there exist two $\dot{\widetilde{\ast}}$-ideals $Q_1 \subsetneq Q_2$ contracting respectively to $P$ and $M$ in $R$. Let $Q$ be a prime ideal of $E$ minimal over $M$ in $E$ such that $Q\sub Q_2$. We want to show that $Q$ has height one. Note that $E$ is $\dot{\widetilde{\ast}}$-Noetherian (Proposition \ref{overnoeth}) and so also strong Mori.
Assume that $M$ is a $t_E$-invertible $t_E$-ideal of $E$. Since $Q$ is minimal over $M$, $Q$ is a $t_E$-prime of $E$ and so $Q\sub N$  for some $t_E$-maximal ideal $N$ of $E$. Now, $M$ is $t_E$-invertible in $E$ and so $ME_N\neq E_N$ is principal. Since $E_N$ is Noetherian (because $E$ is strong Mori), by the Principal Ideal Theorem, $QE_N$ has height one.  It follows that  $Q$ has height-one.

So  $Q\subsetneq  Q_2$ and $Q\cap R=Q_2\cap R=M$, contradicting $\widetilde{\ast}$-INC. It follows that $R$ has $\ast$-dimension one. Since $\ast\leq t$, we have $\astdim(R)=1=\tdim(R)$.
\end{proof}

\begin{cor}\label{corN2}
Let $R$ be a domain and $\ast$ a star operation on $R$. If $R$ is $\widetilde{\ast}$-Noetherian and all the $\ast$-maximal ideals of $R$ are $\ast$-stable, then $\astdim(R)=1=\tdim(R)$.

In particular:
\begin{itemize}
\item[(1)] If  $R$ is Noetherian and $\ast$-stable, then $\astdim(R)=1=\tdim(R)$.

\item[(2)] If $R$ is strong Mori and $t$-stable, then $\tdim(R)=1$.
\end{itemize}
\end{cor}
\begin{proof}
A $\ast$-maximal ideal $M$ of $R$ is $\dot{\ast}$-invertible in $E:=E(M)$; thus it is $t_E$-invertible (note that $\dot{\ast}$ is of finite type since $R$ is $\ast$-Noetherian) and $M = M^{\dot{\ast}} =  M^t$. Thus $M$ is a $t_E$-invertible $t_E$-ideal of $E(M)$ and we can apply Proposition \ref{lemmaN2bis}.

For (1),  note that a Noetherian domain is $\widetilde{\ast}$-Noetherian for every $\ast$. (2) follows for $\ast=t$.
\end{proof}

\begin{cor} \label{N2bis} Let $R$ be a $\widetilde{\ast}$-Noetherian domain. If  $R$ is Clifford $\ast$-regular  and each $\ast$-maximal ideal $M$ of $R$ is a $t_E$-ideal of $E:=E(M)$, then $\astdim(R)=1=\tdim(R)$.
\end{cor}
\begin{proof}
As a consequence of Proposition \ref{prop1}(2), in a Clifford $\ast$-regular $\widetilde{\ast}$-Noetherian domain every $\ast$-maximal ideal $M$ is $t_E$-invertible in $E$.  Hence we can apply Proposition \ref{lemmaN2bis}.
\end{proof}

\begin{remark} \rm
Kabbaj and Mimouni showed that in a pullback diagram of type
$$  \begin{CD}
        R   @>>>    k\\
        @VVV        @VVV    \\
        T  @>\varphi>>  K:= \frac TM\\
        \end{CD}$$
where $T:=(T,M)$ is local Noetherian with maximal ideal $M$, $R$ is a Boole $t$-regular domain if and only if so is $T$ \cite[Proposition 2.3]{KM4}.

In this way it is possible to construct, as in \cite[Example 2.4]{KM4}, examples of Boole $t$-regular local Noetherian domains of $t$-dimension greater than one. Namely, if $T$ is a Noetherian UFD of dimension $n\geq 2$ and $[K:k]$ is finite, then $R$ is a Boole $t$-regular local Noetherian domain whose maximal ideal $M$ is divisorial of height $n$. We note that since $M$ is not divisorial in $T$, then $T=(M:M)=:E(M)$ \cite[Proposition 2.7]{GH} and so,  according to Corollary \ref{N2bis}, $M$ is not a $t_E$-ideal of $E:=E(M)$.
\end{remark}

Next, we prove a technical result that can be applied both to Noetherian $\ast$-regular domains and to strong Mori $t$-regular domains.

\begin{prop} \label{lemmaN1bis}
Let $\ast_1 = \widetilde{\ast_1} \leq \ast_2$ be two star operations on a domain $R$. Assume that $R$ is $\ast_1$-Noetherian of $\ast_1$-dimension one. The following conditions are equivalent for an ideal $I$ of $R$:
\begin{itemize}
\item[(i)] $I$ is Clifford $\ast_2$-regular (respectively, Boole $\ast_2$-regular) ;
\item[(ii)] $I$ is $\ast_2$-stable (respectively, strongly $\ast_2$-stable);
\item[(iii)] $I^{\ast_2}$ is $\dot{\ast_1}$-invertible (respectively, principal) in $E(I^{\ast_2}):=(I^{\ast_2}:I^{\ast_2})$.
\end{itemize}
Hence $R$ is Clifford $\ast_2$-regular  (respectively, Boole $\ast_2$-regular) if and only if $R$ is $\ast_2$-stable (respectively, strongly $\ast_2$-stable).
\end{prop}
\begin{proof} (i) $\ra$ (iii) Assume that $I$ is Clifford $\ast_2$-regular. The overring $E:=E(I^{\ast_2}):=(I^{\ast_2}:I^{\ast_2})$ is $\dot{\ast_1}$-Noetherian by Proposition \ref{overnoeth}. Moreover, $E$ has $\dot{\ast_1}$-dimension one. Indeed, $I^{\ast_2}$ is a $\ast_1$-ideal, so by $\ast_1$-Noetherianity there exists a finitely generated ideal $H$ of $R$ such that $I^{\ast_2} = H^{\ast_1}$. Hence $E = E(H^{\ast_1}) \subseteq R^{[\ast_1]}$ and the $\dot{\ast_1}$-dimension of $E$ is one by \cite[Corollary 4.2]{GP2}. Since $\ast_1$ is of finite type, $\dot{\ast_1}\leq t_E$ and $E$ has $t_E$-dimension one.  Hence
$t_E$-$\Max(E) = \dot{\ast_1}$-$\Max(E)$
and so $t_E$-invertibility coincides with $\dot{\ast_1}$-invertibility. Now $I^{\ast_2}=H^{\ast_1}=H^{\ast_2}$ is $t_E$-invertible in $E$ by Proposition  \ref{prop1}(2). Thus $I^{\ast_2}$ is  $\dot{\ast_1}$-invertible.

 If $I$ is Boole $\ast_2$-regular, we have just seen that $I^{\ast_2}$ is $\dot{\ast_1}$-invertible in $E$. Since $\ast_1 \leq \ast_2$, $I$  is also $\dot{\ast_2}$-invertible.  Hence $I$ is strongly $\ast_2$-stable by Proposition \ref{prop1}(3).

(iii) $\ra$ (ii) is clear since $\ast_1 \leq \ast_2$ and (ii) $\ra$ (i) by Proposition \ref{prop1}(1).
\end{proof}

\begin{cor} \label{lemmaN1}
Let $R$ be a one-dimensional Noetherian domain and $\ast$ a star operation on $R$. The following conditions are equivalent for an ideal $I$ of $R$:
\begin{itemize}
\item[(i)] $I$ is Clifford $\ast$-regular  (respectively, Boole $\ast$-regular);
\item[(ii)] $I$ is $\ast$-stable (respectively, strongly $\ast$-stable);
\item[(iii)] $I^\ast$ is invertible (respectively, principal) in $E(I^\ast)$.
\end{itemize}
Hence $R$ is Clifford  $\ast$-regular (respectively, Boole $\ast$-regular) if and only if $R$ is $\ast$-stable (respectively, strongly $\ast$-stable). In addition, under any of these conditions $\widetilde{\ast}=d$.
\end{cor}
\begin{proof} Apply Proposition \ref{lemmaN1bis} for $\ast_1=d$.

In addition, if $R$ is $\ast$-stable, $\widetilde{\ast}=w$ \cite[Corollary 1.6]{GP} and if $R$ is one-dimensional $w=d$.
\end{proof}

In the local one-dimensional Noetherian case  $\ast$-stability and Boole $\ast$-regularity are equivalent.

\begin{prop} \label{lemmaN1-a}
Let $R$ be a local one-dimensional Noetherian domain and $\ast$ a star operation on $R$. The following conditions are equivalent for an ideal $I$ of $R$:
\begin{itemize}
\item[(i)] $I$ is Clifford $\ast$-regular;
\item[(ii)] $I$ is $\ast$-stable;
\item[(iii)] $I^\ast$ is invertible in $E(I^\ast)$;
\item[(iv)] $I^\ast$ is principal  in $E(I^\ast)$;
\item[(v)]  $I$ is Boole $\ast$-regular.
\end{itemize}
Thus  $R$ is Clifford $\ast$-regular if and only if $R$ is Boole $\ast$-regular if and only if $R$ is (strongly) $\ast$-stable. In addition, under any of these conditions $\widetilde{\ast}=d$.
\end{prop}
\begin{proof} (i) $\lra$ (ii) $\lra$ (iii) and (iv) $\lra$ (v)  by Corollary \ref{lemmaN1}. (v) $\ra$ (i) is clear.

(iii) $\ra$ (iv) Since $R$ is Noetherian, $E:=E(I^\ast)$ is a finitely generated $R$-algebra and $E$ is integral over $R$. By \cite[Chapitre 5, \S 2, n. 1, Proposition 3]{Bou} the number of prime ideals of $E$ contracting to the maximal ideal of $R$ is finite.  Hence $E$ is semilocal and it follows that $I^\ast$ is principal in $E$.

\smallskip
Finally $\widetilde{\ast}=d$ by Corollary \ref{lemmaN1}.
\end{proof}

\begin{remark} \rm  Proposition \ref{lemmaN1-a}  does not hold in the non-local case. In fact, any Dedekind domain that is not a PID furnishes an example of a one-dimensional Noetherian domain that is stable but not  strongly stable.
\end{remark}

For Boole $t$-regularity, the equivalence of conditions (i) and (ii) in the next corollary was proven in \cite[Theorem 2.10]{KM4}.

\begin{cor}
\label{onedim2}  Let $R$ be a strong Mori domain. The following statements are equivalent:
\begin{itemize}
\item [(i)]  $R$ is Clifford $t$-regular (respectively, Boole $t$-regular) of $t$-dimension one;
\item [(ii)] $R$ is $t$-stable (respectively, strongly $t$-stable);
\item[(iii)] $R$ is Clifford $t$-regular (respectively, Boole $t$-regular) and the $t$-maximal ideals of $R$ are $t$-stable.
   \end{itemize}
\end{cor}
\begin{proof}
(i) $\Leftrightarrow$ (ii) follows by Proposition \ref{lemmaN1bis} (for $\ast_1=w$ and $\ast_2=t$) and the fact that a $t$-stable strong Mori domain has $t$-dimension one (Corollary \ref{corN2}).

(ii) $\Rightarrow$ (iii) is obvious and (iii) $\Rightarrow$ (i) follows by Corollary \ref{corN2}.
\end{proof}

By \cite[Proposition 1.6]{GP2}, $w$-regularity and $w$-stability coincide on strong Mori domains, because each $w$-ideal is $w$-finite.  We now give another proof  by using  Proposition \ref{lemmaN1bis}.

\begin{cor}
\label{onedim1}  Let $R$ be a strong Mori domain. The following statements are equivalent:
\begin{itemize}
\item [(i)]  $R$ is Clifford $w$-regular (respectively, Boole $w$-regular);
\item [(ii)] $R$ is $w$-stable (respectively, strongly $w$-stable).
   \end{itemize}
\end{cor}
\begin{proof}
A Clifford $w$-regular strong Mori domain is $t$-stable  (Proposition \ref{mori2}) and so it has $t$-dimension one (Corollary \ref{corN2}). Hence we can apply Proposition \ref{lemmaN1bis} for $\ast_1=\ast_2=w$.
\end{proof}

Motivated by the fact that a Krull Boole $t$-regular domain is a UFD \cite[Proposition 2.2]{KM2}, Kabbaj and Mimouni  ask whether the integral closure of a Noetherian Boole $t$-regular domain is a UFD \cite[Question 2.11(3)]{KM4}. The answer is positive for $w$-regularity (recall that in a Krull domain $t=w$).

\begin{prop} \label{Q2}  Let $R$ be a strong Mori domain. If  $R$ is Boole $w$-regular then $R^{[w]}$ is a UFD.
\end{prop}
\begin{proof}  If $R$ is strong Mori then $R^{[w]}$ is a Krull domain  \cite[Theorem 3.1]{CZ}. In addition, if $R$ is Boole $w$-regular, then $R^{[w]}$ is a GCD-domain \cite[Theorem 4.3]{GP2}. Thus $R^{[w]}$ is a UFD.
\end{proof}

\begin{cor} Let $R$ be a Noetherian  domain. If  $R$ is Boole $w$-regular (respectively, regular) then $R^\prime$ is a UFD (respectively, a PID).

\end{cor}
\begin{proof}
 Just recall that, when $R$ is Noetherian, $R^\prime = R^{[w]}$ and that Clifford regularity implies $d=w$.
 \end{proof}

%%%%%%%%%%%%%%%%%%%%%%%%%%%%%%%%%%%%%%%%%%%%%%%%%%%%%%%%%%%
\section{Mori domains}

 The proof of Proposition \ref{lemmaN2bis} cannot be extended to Mori domains, since it is based on the Principal Ideal Theorem. However we now show that ($w$-)stable Mori domains have ($t$-)dimension one  (cf. \cite[Lemma 4.8]{KM1}).

 \begin{prop} \label{dimwstable} A stable (respectively, $w$-stable) Mori domain has dimension (respectively, $t$-dimension) one.
\end{prop}
\begin{proof} By a corrected version of \cite[Corollary 2.7]{O1}, a local domain $R$ is stable if and only if  one of the following conditions holds:

(a) $R$ is one-dimensional stable;

(b) $R$ is a strongly discrete valuation domain;

(c) $R$ arises from a pullback diagram of type:
$$  \begin{CD}
        R   @>>>    D\\
        @VVV        @VVV    \\
        V  @>\varphi>>   \frac VI\\
        \end{CD}$$
where $V$ is a strongly discrete valuation domain, $I$ is an ideal of $V$, $D$ is a local stable ring of dimension at most one having a prime ideal $P$ such that $P$ contains all the zero-divisors of $D$ and $P^2=(0)$, and $V/I$ is isomorphic to the total quotient ring of $D$ \cite{O4}.

Since a Mori valuation domain is a $DVR$ and in a diagram as in (c)  $R$ is Mori if and only if $V$ is a $DVR$ and $D$ is a field \cite[Theorem 9]{M}, we see that a local Mori stable domain has dimension one.

To conclude, recall that a ($w$-)stable domain is ($t$-)locally stable \cite[Corollary 1.10]{GP}.
\end{proof}

When $\ast$ is the identity, the analog of Proposition \ref{lemmaN1-a} was proved for Mori domains in  \cite{GP2}.

\begin{prop} \label{Moristable}  \cite[Corollary 3.4]{GP2} Let $R$ be a local one-dimensional Mori domain.
The following conditions are equivalent:
\begin{itemize}
\item[(i)] $R$ is Clifford regular;
\item[(ii)] $R$ is (strongly) stable;
\item[(iii)] $R$ is Boole regular.
\end{itemize}
\end{prop}

 In \cite[Theorem 4.8]{GP2} we showed that in $t$-dimension one $w$-regularity and $w$-stability are equivalent on a domain $R$ if and only if the $w$-integral closure $R^{[w]}$ is a Krull domain. For Mori domains we have the following result.

\begin{prop} \label{wmori} Let $R$ be a Mori domain. The following conditions are equivalent:
  \begin{itemize}
  \item [(i)] $R$ is Clifford $w$-regular of $t$-dimension one (respectively, regular of dimension one);
   \item [(ii)] $R$ is Clifford $w$-regular and $R^{[w]}$ is a Krull domain (respectively, regular and $R^\prime$ is a Dedekind domain);
  \item[(iii)] $R$ is $w$-stable (respectively, stable).
\end{itemize}
Under (anyone of) these conditions $R^{[w]}= \widetilde{R}$ (respectively, $R^\prime=\widetilde{R}$) is the complete integral closure of $R$.
\end{prop}
\begin{proof} Since Clifford regularity and stability imply $d=w$, it is enough to prove the theorem for the $w$-operation.

(i) $\ra$ (iii) For each $M\in \tmax(R)$, $R_M$ is one-dimensional Mori and Clifford regular by Proposition \ref{locMori}(2). Hence $R_M$ is stable by Proposition \ref{Moristable}. We conclude by applying Proposition \ref{locMori}(1).

(iii) $\ra$ (ii) If $R$ is $w$-stable, $D:=R^{[w]}$ is a $w_{\vert_D}$-stable overring of $R$ \cite[Corollary 2.6]{GP}, hence $D$ is a  P$v$MD with $t$-finite character such that $(P^2)^t\neq P$, for all $P\in \tspec(R)$ \cite[Theorem 2.9]{GP}. Since $R$ has $w$-dimension one (Proposition \ref{dimwstable}), $D$ has $\dot{w}$-dimension one (see for example \cite[Corollary 4.2]{GP2}) and since $\dot{w}\leq t_D$ it has also $t_D$-dimension one. It follows that $D:=R^{[w]}$ is a Krull domain.

(ii) $\ra$ (i) If $D = R^{[w]}$ is Krull, it has $w_D$-dimension one. Now $R$ is Clifford $w$-regular, so that $\dot{w} = w_D$ \cite[Theorem 4.3]{GP2}. Hence we conclude that $R$ has $w$-dimension one (and so, $t$-dimension one) by \cite[Corollary 4.2]{GP2}.

\smallskip
Since a Krull domain is completely integrally closed, if  $R^{[w]}$ is  Krull, we have $\widetilde{R^{[w]}}=R^{[w]}$. Hence, from $R\sub R^{[w]}\sub\widetilde{R}$, we obtain $\widetilde{R}\sub \widetilde{R^{[w]}}= R^{[w]}\sub \widetilde{R}$ and $R^{[w]}= \widetilde{R}$.
\end{proof}

  If $R$ is strong Mori then $R^{[w]}$ is a Krull domain \cite[Theorem 3.1]{CZ}, hence from Theorem \ref{wmori} we obtain again that $w$-stability and $w$-regularity coincide on strong Mori domains, as seen in Corollary \ref{onedim1}.

Proposition \ref{wmori} improves  \cite[Theorem 4.7]{KM1}.  In relation to this result, Kabbaj and Mimouni ask whether  a local Mori Clifford regular domain is one-dimensional if and only if its complete integral closure $\widetilde{R}$ is Dedekind  \cite[page 633]{KM1}. We can give the following answer.

\begin{cor} \label{Q1} Let $R$ be a Mori Clifford $w$-regular (respectively, regular) domain. The following conditions are equivalent:
  \begin{itemize}
  \item [(i)] $R$ has $t$-dimension one (respectively, dimension  one);
   \item [(ii)]  $R^{[w]}$  is a Krull domain  (respectively, $R^\prime$ is a Dedekind domain).
   \end{itemize}
Under (anyone of) these conditions, $R$ is $w$-stable and $R^{[w]}= \widetilde{R}$ (respectively, $R^\prime=\widetilde{R}$) is the complete integral closure of $R$.
\end{cor}

\begin{prop} Let $R$ be a Mori integrally closed domain. Then the following conditions are equivalent:
  \begin{itemize}
  \item [(i)] $R$ is Clifford $w$-regular;
   \item [(ii)] $R$ is a Krull domain;
  \item[(iii)] $R$ is $w$-stable.
\end{itemize}
\end{prop}
\begin{proof}  (i) $\ra$ (ii)  Since an integrally closed Clifford  $w$-regular domain is a P$v$MD \cite[Corollary 4.5]{GP2}, $R$ is Krull.

(ii) $\ra$ (iii) because in a Krull domain each $t$-ideal is $t$-invertible and $t=w$.

(iii) $\ra$ (i) is clear.
\end{proof}

If $R$ is Mori and \emph{$w$-divisorial}, that is $w=t=v$ (as star operations),  $R$ is  a strong Mori domain of $t$-dimension one \cite[Corollary 4.3]{GE}. If, in addition, $R$ is also $w$-stable, each domain $D$ $t$-linked over $R$ is $w_D$-divisorial \cite[Corollary 3.6]{GP}. Hence from Corollaries \ref{onedim2} and  \ref{onedim1}, we obtain the following result.

\begin{prop} Let $R$ be a $w$-divisorial Mori domain. The following statements are equivalent:
\begin{itemize}
  \item [(i)] $R$ is Clifford $w$-regular;
  \item [(ii)] $R$ is Clifford $t$-regular;
  \item [(iii)] $R$ is $w$-stable;
   \item [(iv)] $R$ is $t$-stable.
    \end{itemize}
In addition, all these properties are inherited by each $t$-linked overring of $R$.
\end{prop}

   	%%%%%%%%%  Bibliography  %%%%%%%%%%%%%%%%%%
%%%%%%%%%%%%%%%%%%%%%%%%%%%%%%%%%%%%%%%%%%%

\end{document}